\documentclass[12pt, twoside]{amsart}
\usepackage{mathrsfs,amsthm,amscd,amssymb, enumitem}
\usepackage{color}
\newcommand{\xr}[1]{\textcolor{red}{#1}}

\newtheorem{thm}{Theorem}[section]

\newtheorem{cor}[thm]{Corollary}
\newtheorem{prop}[thm]{Proposition}

\theoremstyle{definition}
\newtheorem{defn}[thm]{Definition}
\newtheorem{ques}[thm]{Question}
\newtheorem{exmp}[thm]{Example}

\newtheorem{rem}[thm]{Remark}

\newtheorem*{ack}{Acknowledgments}
\newcommand{\Supp}[0]{{\operatorname{Supp}}}

\newcounter{stepnum}
\newcommand{\Step}{%
\par
\refstepcounter{stepnum}%
\textbf{Step \arabic{stepnum}}.\enspace\ignorespaces
}
\title[Some remarks on log surfaces]
{Some remarks on log surfaces}
\author{Haidong Liu}
\date{\xr{2016/12/2, 10:55, version 0.26}}
\address{Department of Mathematics, Graduate School of Science,
Kyoto University, Kyoto 606-8502, Japan}
\email{liu.dong.82u@math.kyoto-u.ac.jp}

\begin{document}
\maketitle
\begin{abstract} 
Fujino and Tanaka established 
the minimal model theory for $\mathbb Q$-factorial 
log surfaces in characteristic $0$ and $p$, respectively. 
We prove that every intermediate surface has only log 
terminal singularities if we run the minimal model program 
starting with a pair 
consisting of a smooth surface and a boundary $\mathbb R$-divisor. 
We further show that such a property does not 
hold if the initial surface is singular. 
\end{abstract}

\section{Introduction}\label{sec1}
We work over an algebraically closed field of arbitrary characteristic
throughout this paper.
We will also follow the language and notational conventions of the book \cite{km}
unless stated otherwise. 

Let $(X,\Delta)$ be a log surface.
Remember that a pair $(X,\Delta)$ is called {\em{log surface}} if
$X$ is a normal algebraic surface and $\Delta$ is a boundary
$\mathbb R$-divisor on $X$ such that 
$K_{X}+\Delta$ is $\mathbb R$-Cartier.
To complete Fujita's results \cite{ft} on the semi-ampleness of 
semi-positive parts of Zariski decompositions of log canonical divisors 
and the finite generation of log canonical rings for smooth projective log surfaces,
Fujino \cite{fujino} developed the log minimal model program for projective log surfaces
in characteristic $0$.
It is generalized to characteristic $p>0$ by Tanaka in his paper \cite{tnk}.
One of their main results is the following:

\begin{thm}[{\cite[Theorem 3.3]{fujino}, \cite[Theorem 1.1]{tnk}}] \label{thm1.1}
Let $(X,\Delta)$ be a log surface which is not necessarily log canonical,
and let $\pi: X \rightarrow S$ be a projective morphism onto an
algebraic variety $S$. Assume that $X$ is $\mathbb Q$-factorial.
Then we can run the log minimal model program over $S$ with respect 
to $K_{X}+\Delta$ and get a sequence of at most $\rho(X/S)-1$ contractions
$$
(X,\Delta)=(X_{0},\Delta_{0}) \rightarrow (X_{1},\Delta_{1}) 
\rightarrow
\cdots
\rightarrow (X_{k},\Delta_{k})=(X^{*},\Delta^{*})
$$
over $S$ such that one of the following holds:
\begin{itemize}
\item[(1)] (Minimal model) $K_{X^{*}}+\Delta^{*}$ is nef over $S$. In this case,
$(X^{*},\Delta^{*})$ is called a minimal model of $(X,\Delta)$.
\item[(2)] (Mori fiber space) There is a morphism 
$g: X^{*} \rightarrow C$ over $S$ such that $-(K_{X^{*}}+\Delta^{*})$
is g-ample, $\rm{dim}C<2$, and $\rho(X^{*}/C)=1$. We sometimes call
$g:(X^{*},\Delta^{*}) \rightarrow C$ a Mori fiber space.
\end{itemize}
Note that $X_{i}$ is $\mathbb Q$-factorial for every $i$.
Furthermore, if $K_{X}+\Delta$ is big, then on the minimal model
$(X^{*},\Delta^{*})$,
 $K_{X^{*}}+\Delta^{*}$ is nef and big over $S$.
\end{thm}

First, we try to clarify that, given such a log surface $(X,\Delta)$ where 
$X$ is smooth, 
what every intermediate surface $X_i$ would look like
after running this log minimal model program.
Note that the final log surface $(X^{*},\Delta^{*})$ could be a minimal model or
a Mori fiber space $g:(X^{*},\Delta^{*}) \rightarrow C$.
The following theorem is our main result in this paper to achieve this aim.

\begin{thm}[Theorem \ref{thm3.1}]\label{thm1.2}
 Notations are as in Theorem\ref{thm1.1}. If $X$ is smooth and the 
 coefficients of $\Delta$ are $\leq 1-\varepsilon$, then
 $X_i$ is $\varepsilon$-log terminal for every $i$.
In particular,  $X^{*}$ is $\varepsilon$-log terminal.
\end{thm}

Next, a natural question is that, 
given a log surface $(X,\Delta)$ where 
$X$ is not smooth, 
what every intermediate surface $X_i$ would look like
after running log minimal model program.

\begin{prop}\label{prop1.3} 
In Theorem \ref{thm1.1}, 
$X_i$ is not always log canonical even if 
$X$ is log canonical.
\end{prop}

Moreover, we have: 

\begin{prop}\label{prop1.4}
In Theorem \ref{thm1.1}, 
$X_i$ is not always log canonical 
even if $X$ is $\varepsilon$-log canonical and the coefficients of 
$\Delta$ are $\leq 1-\varepsilon$ for some $0<\varepsilon<1$.
\end{prop}

In Section \ref{sec4} we construct some examples to show that
Proposition \ref{prop1.3}, \ref{prop1.4} are true.
Furthermore, we show that $X_i$ could not even be MR log canonical if $X$ is not smooth.
In fact this shows that Fujino and Tanaka's minimal model program on log surfaces is more general 
than Alexeev's minimal model program which is running  mainly on MR log canonical surfaces in 
\cite[Section 10]{alex}
(see Definition \ref{def2.2} for the definition of MR log canonical).

\begin{ack}
The author would like to thank professor Fujino for so many inspirational suggestions and comments.
The author would like to thank H.Tanaka for his helpful comments too.
He would also like to thank Chen Jiang for many discussions on MR log canonical when they were attending
in the conference of HDAG held in Utah. 
\end{ack}

\section{Preliminaries}\label{sec2}
Let $(X,\Delta)$ be a log surface. 
If $X$ is smooth, then it is $\mathbb Q$-factorial. 
Choose a set $I \subset [0,1-\varepsilon]$ where
$\varepsilon \in [0,1]$ is a fixed real number.
Assume that the coefficients of $\Delta$ are in $I$.
Remember that a set $I$ of real number satisfies the 
{\em{descending chain condition}} or DCC, 
if it does not contain any infinite strictly decreasing sequence. 
Finally, recall that
the {\em{volume}} of an $\mathbb R$-divisor $D$ on a normal projective variety 
$X$ of dimension $n$ is defined as 
$$
{\rm vol}(D)= \limsup_{m\to \infty} 
\frac{h^{0}(\lfloor mD \rfloor)}{m^{n}/n!}
$$

We recall some kinds of singularities and  MR singularities following the same way of Alexeev. 

\begin{defn}[{\cite[Definition 1.5]{alex}}]\label{def2.1}
Let $(X,\Delta)$ be a log surface. 
Fixed a small non-negative real number $\varepsilon$, it is called:
\begin{itemize}
\item[1,] $\varepsilon$-log canonical, if the total discrepancies  $\geq -1+\varepsilon$
\item[2,]  $\varepsilon$-log terminal, if the total discrepancies  $> -1+\varepsilon$
\end{itemize}
for every resolution $f: Y \rightarrow X$. Simply, we call it $\varepsilon$-lc
or  $\varepsilon$-lt instead. Note that when $\varepsilon$ is not zero,
we can replace $\varepsilon$ by a smaller positive $\varepsilon'$,
and assume that $\varepsilon$-log canonical 
is $\varepsilon'$-log terminal.
\end{defn}

\begin{defn}[{\cite[Definition 1.7]{alex}}]\label{def2.2}
We call a log surface $(X,\Delta)$ MR log canonical,
MR $\varepsilon$-log canonical, MR $\varepsilon$-log terminal etc.
if we require the previous inequalities in Definition \ref{def2.1}
to hold not for all resolutions
$f: Y \rightarrow X$ but only for a distinguished one,
the minimal desingularization.
\end{defn}

A strange but trivial example of MR log canonical log surface is the following:

\begin{exmp}[]\label{exmp2.3}
Given a log surface $(X,\Delta)$, where $X$ is smooth and
$\Delta$ is a boundary. $(X,\Delta)$ is not necessarily log canonical 
in the usual sense.
But $id: X \rightarrow X$ is the minimal desingularization,
therefore $(X,\Delta)$ is MR log canonical.
\end{exmp}

\section{Main results}\label{sec3}

Now we go to the proof of Theorem \ref{thm1.2}.
Note that $\varepsilon$ in this theorem could be zero:

\begin{thm}[]\label{thm3.1}
Notations are as in Theorem \ref{thm1.1}. If $X$ is smooth and the 
coefficients of $\Delta$ are $\leq 1-\varepsilon$, then
 $X_i$ is $\varepsilon$-log terminal for every $i$.
In particular,  $X^{*}$ is $\varepsilon$-log terminal.
\end{thm}

\begin{proof}
\Step
Run log minimal model program
on $K_{X}+\Delta$ as in Theorem \ref{thm1.1}:
$$
(X,\Delta)=(X_{0},\Delta_{0}) \rightarrow (X_{1},\Delta_{1}) 
\rightarrow
\cdots
\rightarrow (X_{k},\Delta_{k})=(X^{*},\Delta^{*})
$$
where $(X^{*},\Delta^{*})$ is a minimal model or a Mori fiber space.
In the following proof, we consider everything over $X_j$ for a fixed $j$.
Put $X^\dag=X_j$ for this fixed $j$.
Then take $X^\dag$ as a base 
(if needed, shrink $X^\dag$ to be affine since 
$\varepsilon$-log terminal or not is a local property)
and run $(K_{X}+\Delta)$-LMMP on
the relative morphism $f: X \rightarrow X^\dag$, which ends up again on $X^\dag$
and $K_{X^\dag}+\Delta^\dag$ is nef over $X^\dag$.
Each step we have a relative morphism $ X_{i} \rightarrow X^\dag$ ($i \leq j$)
and denote it by $X_{i}/X^\dag$.
We use $f_{i}$ and $h_{i}$ to denote 
the morphisms $(X_{i},\Delta_{i})/X^\dag \rightarrow (X^\dag,\Delta^\dag)/X^\dag$
and $ (X,\Delta)/X^\dag \rightarrow (X_{i},\Delta_{i})/X^\dag$
with that $h_{j}=f_{0}=f$.
By \cite[Section 3]{fujino} and \cite[Section 3]{tnk},
$$
K_{X_{i}}+\Delta_{i}=f_{i}^{*}(K_{X^\dag}+\Delta^\dag)+ E_{i}
$$
where $E_{i}$ are all effective
over $X^\dag$ for every $0 \leq i < j$.
In particular,
$h_{i*}(K_{X}+\Delta)=K_{X_{i}}+\Delta_{i}$,
$h_{i*}(\Delta)=\Delta_{i}$.
Furthermore, every curve in $\mathrm{Exc}(f)=\Supp(E_{0})$ is a
smooth rational curve by \cite[Proposition 3.8]{fujino} and \cite[Theorem 3.19]{tnk}.

\bigskip \Step
Now we may assume that 
there is no $(-1)$-curve in $\mathrm{Exc}(f)$. 
Indeed, if there is some $(-1)$-curve, say $C$, in  $\mathrm{Exc}(f)$,
then by Castelnuovo's theorem, contracting this $(-1)$-curve in $X/X^\dag$
leads to a new smooth surface $X'/X^\dag$.
Therefore we can run 
another $(K_{X'}+\Delta')$-LMMP 
over $X^\dag$ until reaching to a final log surface
$(\widetilde{X},\widetilde{\Delta})/X^\dag$,
where $\Delta'$ is the image of  $\Delta$.
Every assumption of $(X,\Delta)$ is obviously keeping if we replace 
$(X,\Delta)$ by $(X',\Delta')$ except that we need to prove
$(\widetilde{X},\widetilde{\Delta}) \cong (X^\dag,\Delta^\dag)$.
We have three morphisms over $X^\dag$:
$\pi: X \rightarrow X'$, $g: X'\rightarrow \widetilde{X}$ and
 $\rho:\widetilde{X} \rightarrow X^\dag$ 
such that
$$
K_{X}+\Delta=\pi^{*}(K_{X'}+\Delta')+aC
$$ 
$$
K_{X'}+\Delta'=g^{*}(K_{\widetilde{X}}+\widetilde{\Delta})+ E'_{0}
$$
$$
K_{\widetilde{X}}+\widetilde{\Delta}=\rho^{*}(K_{X^\dag}+\Delta^\dag)+ D
$$
where 
$\pi: X \rightarrow X'$ is the Castelnuovo's contraction, 
$\rho$ is not necessarily 
the identity and $K_{\widetilde{X}}+\widetilde{\Delta}$ 
is nef over $X^\dag$. Then by negativity 
lemma (see \cite[Lemma 3.39 and Lemma 3.40]{km}),
we have that $-D \geq 0$, since $K_{\widetilde{X}}+\widetilde{\Delta}-D\sim_{\rho} 0$
and $D$ is $\rho$-exceptional.
Remember that $K_{X}+\Delta=f^{*}(K_{X^\dag}+\Delta^\dag)+ E_{0}$,
$f^{*}=\pi^{*}g^{*}\rho^{*}$. That is, 
$E_{0}\sim_{f} \pi^{*}g^{*}D + \pi^{*}E'_{0}+aC$. 
By negativity lemma again, $D > 0$
since $E_{0}$ is effective and both sides have the same support.
Therefore we get a contradiction
unless $\rho$ is an identity. That is,
$(\widetilde{X},\widetilde{\Delta}) \cong (X^\dag,\Delta^\dag)$.
Then, by contracting 
$(-1)$-curves finitely many times, we may assume that 
$\mathrm{Exc}(f)$ contains no $(-1)$-curve from now on.

\bigskip \Step
Assume that $C_{i}$ is the contracted curve in step $i$ of the 
log minimal model program, then 
$(K_{X_{i}}+\Delta_{i})\cdot C_{i}<0$. Therefore 
$$
(K_{X}+\Delta)\cdot h_{i}^{*}(C_{i})=(K_{X_{i}}+\Delta_{i})\cdot C_{i}<0
$$
Note that $(h_{i}^{*}(C_{i}))^{2}=(C_{i})^{2}<0$ by the negativity lemma.
Then $K_{X} \cdot h_{i}^{*}(C_{i}) \geq 0$ since $h_{i}^{*}(C_{i})$ is effective
and its support contains no $(-1)$-curve. 
Indeed, if 
$K_{X} \cdot h_{i}^{*}(C_{i}) < 0$, there must be a curve, say,
$E$, in $\Supp h_{i}^{*}(C_{i})$ such that 
$K_{X} \cdot E < 0$. But $E^{2}< 0$ since $E$ is in $\mathrm{Exc}(f)$.
Thus it is a 
 $(-1)$-curve which contradicts our assumption.
Therefore $\Delta \cdot h_{i}^{*}(C_{i})<0$. Then
$$
\Delta_{i} \cdot C_{i}=h_{i*}(\Delta)\cdot C_{i}=\Delta \cdot h_{i}^{*}(C_{i})<0
$$
That is, $C_{i}$ is in $\Supp \Delta_{i}$, and its strict transform
is in $\Supp \Delta$.
Therefore all those curves in $\mathrm{Exc}(f)$ must be such a 
strict transform of $C_{i}$ under the assumption of the above step.

\bigskip \Step
Next, we need to prove that, for the resolution $f: X \rightarrow X^\dag$ where 
$K_{X}=f^{*}K_{X^\dag}+ \sum a_{i}F_{i}$, we have that $ a_{i} > -1+\varepsilon$.
Note that $ K_{X}+\Delta=f^{*}(K_{X^\dag}+\Delta^\dag)+ E_{0}$ where
$E_{0}$ is effective in $\mathrm{Exc}(f)$ 
and $F_{i}$ is in $\Supp \Delta$
by the above steps. Furthermore, let $\Delta=\sum\delta_{i}F_{i}+ \Delta'$ where
$\sum F_{i}$ and $\Delta'$ have no common components.
Therefore, $f_{*}\Delta'=\Delta^\dag$. Then
$$
 K_{X}+\Delta=f^{*}K_{X^\dag}+\sum a_{i}F_{i}+\sum \delta_{i}F_{i}+ \Delta'=f^{*}K_{X^\dag}+f^{*}\Delta^\dag+ E_{0}
$$
That is, 
$$
\sum(a_{i}+\delta_{i})F_{i}=f^{*}\Delta^\dag-\Delta'+ E_{0}
$$
in which both sides are supported in $\mathrm{Exc}(f)$ and 
the right hand side is effective.
Thus comparing both sides, $a_{i}+\delta_{i}>0$. 
That is, $a_{i} > -\delta_{i}\geq -1+\varepsilon$ since 
the coefficients of $\Delta$ are $\leq 1-\varepsilon$.

Finally, we claim that, 
the resolution $f: X \rightarrow X^\dag$ is a log resolution. That is,
the reduced $\sum F_{i}$ must be a simple normal crossing curve. 
We can prove this claim by \cite[Theorem 4.7]{km} and the above steps,
which is pointed out by Tanaka.
But here we use a different way.
Remember that
$F_{i}$ are all smooth extremal rational curves since 
$X^\dag$ has rational singularities by \cite[Theorem 6.2]{fjtnk} for any characteristic.
Furthermore, the dual graph of $\sum F_{i}$ must be a tree.
This shows that the reduced $\sum F_{i}$ must be a simple normal crossing curve.
We get what we want. 
\end{proof}

From the above theorem, we know that when $X$ is smooth,
those contracting curves in log 
minimal model program consist of some images of $(-1)$-curves and some
components of $\Supp \Delta$.
Several direct but important implications of Theorem \ref{thm3.1} are the following.
When $K_{X}+\Delta$ is big, $K_{X^{*}}+\Delta^{*}$ is nef and big
on the minimal model.
What we have done in the proof of Theorem \ref{thm3.1} is in fact showing that 
$f: X' \rightarrow X^{*}$ is exactly the minimal 
desingularization 
and  $(X^{*},\Delta^{*})$ is MR $\varepsilon$-log terminal.
Then the following corollaries are just simple consequences
of \cite[Theorem 7.6, Theorem 7.7, Theorem 8.2]{alex}.
It is another way to see that Fujino and Tanaka's mimimal model program 
on log surfaces cover Alexeev's mimimal model program stated in 
\cite[Section 10]{alex}.

\begin{cor}[] \label{cor3.2}
Let $(X,\Delta)$ be a projective log surface where
$X$ is smooth and $K_{X}+\Delta$ is big.
Fixing $\varepsilon>0$, let $I\subset [0,1-\varepsilon]$ be a {\rm DCC} set
and the coefficients of $\Delta$ be in $I$.
If there is a positive integer $M$ such that $(K_{X^{*}}+\Delta^{*})^{2}\leq M$
where $(X^{*},\Delta^{*})$ is a minimal model of $(X,\Delta)$,
then these $(X^{*}, \Supp \Delta^{*})$ belong to a bounded family.
\end{cor}

\begin{cor}[] \label{cor3.3}
Let $(X,\Delta)$ be a projective log surface where
$X$ is smooth and $K_{X}+\Delta$ is big.
Fixing $\varepsilon \geq 0$, let $I\subset [0,1-\varepsilon]$ be a {\rm DCC} set
and the coefficients of $\Delta$ be in $I$.
Then $(K_{X^{*}}+\Delta^{*})^{2}$ is a DCC set.
In particular the volume ${\rm vol}(K_{X}+\Delta)$
is bounded from below away from $0$.
\end{cor}
\begin{proof}
Since ${\rm vol}(K_{X}+\Delta)={\rm{vol}}(K_{X}^{*}+\Delta^{*})=(K_{X^{*}}+\Delta^{*})^{2}$ by 
Theorem \ref{thm3.1}, this corollary is  a direct consequence of 
\cite[Theorem 8.2]{alex}.
\end{proof}

\begin{rem}[] \label{rem3.4}
Note that in Corollary \ref{cor3.2},
the $\varepsilon$ is smaller, the bounded family of
$(X^{*},\Supp \Delta^{*})$ is bigger. 
When $\varepsilon$ goes to $0$, all those $X^{*}$ may not be in a bounded family,
so not be $(X^{*},\Supp \Delta^{*})$.
See  \cite[Remark 1.5]{lin} for the example showing that  $X^{*}$ could be 
$\mathbb Q$-Fano and not in a bounded family. 
Note also that Corollary \ref{cor3.3} is an answer of the question coming 
from the first version of Di Cerbo's paper 
\cite[Question 4.3]{dic2} which has been confirmed by his second version.
\end{rem}

\section{Examples}\label{sec4}

By \cite[Section 10]{alex}, we easily see that
if the log surface $(X,\Delta)$ is MR $\varepsilon$-log canonical,
then so is every $(X_{i},\Delta_{i})$
in the step of log minimal model program;
by Grothendieck spectral sequence, 
it is also easy to see that if $X$ has only rational singularities,
then so has every  $X_{i}$. Now it is natural to 
generalize Theorem \ref{thm3.1} and ask that
if $X$ is $\varepsilon$-log canonical,
is so every $X_i$ or not.
But unfortunately we have the following example:

\begin{exmp}[]\label{exmp4.1}
There is a well known example of log canonical surface.
In fact, it is rational but not log terminal.
Blowing up at a point of $\mathbb P ^{2}$, we get a $(-1)$-curve $E_{0}$;
find three points at $E_{0}$ and blow up 
several times (at these three points and some points at the exceptional curves over them),
we can easily get a surface $Y$ and
four smooth rational curves $E_{0}$, $E_{1}$, $E_{2}$, $E_{3}$ on it such that
$n_{0}=-E_{0}^{2}\geq 3$, $n_{1}=-E_{1}^{2}=2$, $n_{2}=-E_{2}^{2}=3$, $n_{3}=-E_{3}^{2}=6$
where by abusing of notations, we still use $E_{0}$ to denote its strict transform
on $Y$. By construction,
$E_{i}\cdot E_{0}=1$, $E_{i}\cdot E_{j}=0$ where $i, j=1,2,3$.
Let $E=E_{0}+E_{1}+E_{2}+E_{3}$, then its dual graph is a triple fork.
See also that its
intersection matrix is negative definite. Therefore by Artin's criterion \cite{at},
we can contract $E$ and finally get a surface $X$ with a singular point.
Now we have $f: Y \rightarrow X$ with $K_{Y}=f^{*}K_{X}+ \sum a_{i}E_{i}$.
Using adjunction, we have that:
$$
-2+n_{0}=-a_{0}n_{0}+a_{1}+a_{2}+a_{3};
$$
$$
0=-2+n_{1}=a_{0}-a_{1}n_{1}=a_{0}-2a_{1};
$$
$$
1=-2+n_{2}=a_{0}-a_{2}n_{2}=a_{0}-3a_{2};
$$
$$
4=-2+n_{3}=a_{0}-a_{3}n_{3}=a_{0}-6a_{3}.
$$
Solve these equations we have $a_{0}=-1$, $a_{1}=-\frac{1}{2}$,
$a_{2}=-\frac{2}{3}$, $a_{3}=-\frac{5}{6}$.
These show that 
the singularity of $X$ is exactly log canonical but not log terminal.
Keeping this example in mind, we construct an example as following:

Similar to the above blowing-up method, we can easily construct 
a surface $Y$ and 
five smooth rational curves $D$, $E_{0}$, $E_{1}$, $E_{2}$, $E_{3}$ on it such that
$n=-D^{2}$ is as big as we want, $n_{0}=-E_{0}^{2}\geq 3$,
$n_{1}=-E_{1}^{2}=2$, $n_{2}=-E_{2}^{2}=3$, $n_{3}=-E_{3}^{2}=6$
and 
$E_{i}\cdot E_{0}=D \cdot E_{0}=1$, $E_{i}\cdot E_{j}= D\cdot E_{i}=0$ where $i, j=1,2,3$.
Let $E=E_{0}+E_{1}+E_{2}+E_{3}$ and $F=E+D$. Then $E$ is a triple fork
and $F$ is a quadruple fork in dual graph. Note that both of 
the intersection matrices of 
$E$ and $F$ are negative definite.
By contracting $E$ on $Y$ we get a morphism $f$ from $Y$ 
to a log canonical surface $X$ which is rational
but not log terminal as above. Now consider the log surface $(X, D')$ 
where $D'$ is the image of $D$.
 $D'$ is still a smooth rational curve by construction since $E \cdot D=1$.
Note that $(K_{X}+D')\cdot D'< 0$. Indeed, Let 
$f^{*}D'=D+\sum c_{i}E_{i}$. Then by $E_{i}\cdot f^{*}D' =0$,
$$
0=1-c_{0}n_{0}+c_{1}+c_{2}+c_{3};
$$
$$
0=c_{0}-c_{1}n_{1};
$$
$$
0=c_{0}-c_{2}n_{2};
$$
$$
0=c_{0}-c_{3}n_{3}.
$$
That is, $c_{0}=\frac{1}{n_{0}-1}$, $c_{1}=\frac{c_{0}}{2}$, 
$c_{2}=\frac{c_{0}}{3}$, $c_{3}=\frac{c_{0}}{6}$. Then 
$$
(K_{X}+D')\cdot D'=(K_{Y}+D)\cdot f^{*}D'=(K_{Y}+D)\cdot (D+\sum c_{i}E_{i})
$$
$$
=(K_{Y}+D) \cdot D+\sum c_{i} (K_{Y}\cdot E_{i})+\sum c_{i} (D\cdot E_{i})
$$
$$
= -2+c_{0} (-2+n_{0})+c_{1}(-2+2)+c_{2}(-2+3)+c_{3}(-2+6)+ c_{0}
$$
$$
= -2+c_{0} (-2+n_{0})+\frac{c_{0}}{3}+\frac{2c_{0}}{3}+ c_{0}=-2+1+c_{0}=c_{0}-1< 0
$$
since $c_{0}=\frac{1}{n_{0}-1}<1$.
Now contracting $D'$ on $X$ by log minimal model program,
we get a log surface $(X^{*}, 0)$ where $X^{*}$ is no longer
log canonical since the dual graph of $F$ is a quadruple fork
which is not in the classification of dual graph of 
log canonical singularities in \cite[Theorem 4.7]{km}.
Furthermore, it is not even MR log canonical by calculating 
the discrepancy of $E_{0}$.
But remember that 
$X^{*}$ still has  rational singularities. 
\end{exmp}

\begin{exmp}[]\label{exmp4.2}
We just gave an example for Proposition \ref{prop1.3}
where $\varepsilon=0$. In fact, by 
a similar construction as above, we can get some examples
where $\varepsilon>0$. 
A sketch of construction is the following.
As Example \ref{exmp4.1},
we can easily construct 
a surface $Y$ and 
five smooth rational curves $D$, $E_{0}$, $E_{1}$, $E_{2}$, $E_{3}$ on it with
$n=-D^{2}$, $n_{i}=-E_{i}^{2}$ such that $n=3$, $n_{0}=5$, $n_{1}=n_{2}=n_{3}=2$.
Let $E=E_{0}+E_{1}+E_{2}+E_{3}$ and $F=E+D$. Then $E$ is a triple fork
and $F$ is a quadruple fork which is not in the classification of dual graph of 
log canonical singularities. Note that both of 
the intersection matrices of 
$E$ and $F$ are negative definite.
Choose an $\varepsilon$ such that $0<\varepsilon\leq \frac{1}{7}$. 
The same calculation as Example \ref{exmp4.1} shows that 
by contracting $E$ on $Y$ we get a morphism $f$ from $Y$ 
to an $\varepsilon$-log canonical surface $X$. Now consider the log surface $(X, bD')$ 
where $D'$ is the image
of $D$. Note that $D'$ is still a smooth rational curve by construction.
Choose a proper real number $b$ such that $(K_{X}+bD')\cdot D'< 0$.
By careful calculations as in Example 
\ref{exmp4.1}, we can check that $(K_X+bD')\cdot D'<0$ for 
$b>\frac{13}{19}$. Therefore,
$(K_X+(1-\varepsilon)D')\cdot D'<0$ for $0<\varepsilon \leq \frac{1}{7}$.
%%\lq\lq\rq\rq
Now contracting $D'$ on $(X, bD')$ by log minimal model program,
we get a log surface $(X^{*}, 0)$ where $X^{*}$ is no longer
log canonical.
This gives an example to confirm Proposition \ref{prop1.4}.
\end{exmp}

\begin{rem}[]\label{rem4.3}
The above two examples are based on one of the dual graphs 
of log canonical singularities in \cite[Theorem 4.7]{km}. In fact, we can construct
similar examples based on the other dual graphs there and get a bunch of similar examples.
\end{rem}

It will be interesting to ask the following question:
\begin{ques}[]\label{ques4.4}
In Theorem \ref{thm1.1}, if $X$ is canonical, 
is $X_i$ log canonical?
\end{ques}

%%%%%%%%%%%%%%%%%%%%%%%%


\begin{thebibliography}{Amb03}


\bibitem[Alex94]{alex} V.~Alexeev, Boundedness and $K^{2}$ for log surfaces.
Internat. J. Math. 5 (1994), no.6, 779-810.

\bibitem[Art62]{at} M.~Artin, 
Some numerical criteria for contractability of curves on algebraic surfaces. 
Amer. J. Math. 84 (1962), 485–496.


\bibitem[Dic16]{dic2} G.~Di Cerbo, On Fujita's log spectrum conjecture.
arXiv:1603.09315v1.

\bibitem[FT12]{fjtnk} O.~Fujino, H.~Tanaka, On log surfaces. 
Proc. Japan Acad. Ser. A Math. Sci.
Volume 88, Number 8 (2012), 109-114.

\bibitem[Fjn10]{fujino} O.~Fujino, Minimal model theory for log surfaces.
Publ. Res. Inst. Math. Sci {\textbf{48}} (2012), no.2, 339-371.

\bibitem[Fjt84]{ft} T.~Fujita, Fractionally logarithmic canonical rings of algebraic surfaces. 
J. Fac. Sci. Univ. Tokyo Sect. IA Math. 30 (1984), no. 3, 685–696.

\bibitem[KM98]{km} J.~Koll\'ar, S.~Mori, Birational geometry of algebraic varieties, Cambridge
Tracts in Mathematics, Vol. 134, 1998.

\bibitem[Lin03]{lin} J.~Lin, Birational unboundedness of $Q$-Fano threefolds.
Int.Math. Res. Not. 6 (2003), 301-312.

\bibitem[Tnk12]{tnk} H.~Tanaka, 
Minimal models and abundance for positive characteristic log surfaces.
Nagoya Math. J. Volume 216 (2014), 1-70.


\end{thebibliography}
\end{document}